%% file: adrv6_arxiv.tex
\documentclass[letterpaper, 10ptxs]{IEEEtran}
\IEEEoverridecommandlockouts
\usepackage{amsmath,amssymb,amsthm, bbm}
\usepackage{graphicx,caption,subcaption,multirow}
\usepackage{enumerate,xcolor,xfrac,url}

\newtheorem{theorem}{Theorem}
\newtheorem{lemma}{Lemma}

\newtheorem{corollary}{Corollary}

\input{defs}

\usepackage{array}
\newcolumntype{H}{>{\setbox0=\hbox\bgroup}c<{\egroup}@{}}

\title{\LARGE \bf
A near-optimal maintenance policy for automated DR devices
}
\author{
Carlos Abad and Garud Iyengar
\thanks{Supported in part by DOE grant DE-AR0000235, ONR grant N000140310514, and NSF grant DMS-1016571}%
\thanks{Both authors are with the Industrial Engineering and Operations Research Department, Columbia University,
        New York, NY 10027, USA
        {\tt\small ca2446@columbia.edu,garud@ieor.columbia.edu}}%
}

\begin{document}
\maketitle
\begin{abstract}
  Demand side participation is now widely recognized as being extremely
    critical for satisfying the 
  growing 
  electricity demand   in the US. 
  The primary mechanism for demand management in the US is demand
  response (DR) programs  that  
  attempt to reduce or shift demand by giving incentives to 
  participating customers 
  via price discounts
  or 
  rebate payments. 
  Utilities that offer DR programs rely on automated DR devices (ADRs)
  to automate the response to DR
  signals. The ADRs are faulty; 
  but the working state of the ADR is not directly
  observable -- one can, however, attempt to infer it from the power consumption
  during DR events. The utility loses revenue when a malfunctioning
  ADR does not respond to a DR signal; however, sending a maintenance crew
  to check and reset the ADR 
  also incurs costs. In this paper, we 
  show that
  the  problem of maintaining
  a pool of ADRs using a limited number of maintenance crews
  can be formulated as a restless bandit problem, and that one can compute a
  near-optimal policy for
  this problem using Whittle indices. 
  We  show that
  the  Whittle
  indices can be efficiently computed using a variational
  Bayes procedure even when the
  load-shed magnitude is noisy and when there is a random 
  mismatch
  between the clocks at the utility and at the meter.
  The results of our numerical experiments
  suggest that the Whittle-index based approximate policy is 
  within $3.95$\% 
  of the optimal solution for all reasonably low values of
  the signal-to-noise ratio in the meter readings.
\end{abstract}


\section*{Nomenclature}

\addcontentsline{toc}{section}{Nomenclature}
\subsection*{Indices}
\begin{IEEEdescription}[\IEEEusemathlabelsep\IEEEsetlabelwidth{ }]
\item[$t$] Index of DR events
\item[$k$] Index of discrete points in belief space $[0,1]$
\item[$i$] Index of meter readings during DR event
\item[$j$] Index of samples of meter reading vectors
\end{IEEEdescription}

\subsection*{Constants and parameters}
\begin{IEEEdescription}[\IEEEusemathlabelsep\IEEEsetlabelwidth{ }]
  \item[$\lambda$] Expected dollar savings for utility when ADR is 
  working
  \item[$c$] Cost incurred by utility for sending a repair crew
  \item[$\theta$] Customer compensation for participating in a DR 
  event
  \item[$p$] Prior probability of ADR failure
  \item[$\cS$] ADR state space
  \item[$\gamma_0$] Non-operational ADR state
  \item[$\gamma_1$] Operational ADR state
  \item[$\cA$] Set of actions available to utility before a DR event
  \item[$\alpha_0$] Do-nothing action
  \item[$\alpha_1$] Send-crew-to-reset-ADR action
  \item[$\cX$] Set of possible meter reading vectors during DR event
  \item[$\beta$] Discount factor between DR events
  \item[$n$] Discretization size of belief space $[0,1]$
  \item[$D$] Number of ADRs managed by utility
  \item[$M$] Number of repair crews overseeing ADRs
  \item[$m$] Number of meter readings during DR event
  \item[$N$] Number of samples of meter reading vectors
  \item[$\by$] Vector of estimated normal power consumption
  \item[$\sigma$] Standard deviation of load estimation residuals
  \item[$\nu_0$] Expected load-shed during a DR event
  \item[$\eta_0$] Load-shed precision (inverse of variance)
  \item[$d$] Maximum absolute meter-utility clock mismatch
\end{IEEEdescription}

\subsection*{Variables}
\begin{IEEEdescription}[\IEEEusemathlabelsep\IEEEsetlabelwidth{ }]
\item[$s$] State of ADR before DR event
\item[$a$] Action taken by utility before DR event
\item[$\bx$] Vector of meter readings during DR event
\item[$b$] Belief probability that ADR is operational
\item[$\tilde{\varepsilon}$] Load estimation residual
\item[$r$] Actual load-shed during a DR event
\item[$\delta$] Actual mismatch between meter and utility clocks
\item[$\bfomega$] Low-discrepancy sample in unit hypercube
\item[$\bz$] Vector of IID standard normal samples
\end{IEEEdescription}

\subsection*{Functions}
\begin{IEEEdescription}[\IEEEusemathlabelsep\IEEEsetlabelwidth{$P_{ss^\prime}(a)
 \hspace{-6pt}$}]
\item[$R_{as}$] Utility profit when taking action $a$ in ADR state $s$
\item[$P_{ss^\prime}(a)$] ADR transition probability
\item[$Q_{sa}(\bx)$] Meter reading observation conditional probability
\item[$r_{a}(b)$] Expected profit under action $a$ in belief state $b$
\item[$\Gamma_{a\bx}(b)$] Belief probability transition map
\item[$W_{ab}(\bx)$] Up-to-date meter reading observation probability
\item[$V(b)$] POMDP value function
\item[$\Phi(\cdot)$] Standard normal cumulative distribution function
\item[$\varphi(\cdot)$] Standard normal density function
\end{IEEEdescription}

\section{Introduction}
%
Demand Response (DR) refers to a set of activities where end-use customers
change or shift their normal electricity consumption patterns 
in
response to changes in the
price of electricity or other incentive payments 
\cite{FERC}. 
DR has been extensively 
deployed 
over the past several years 
to   
improve electric grid reliability and market efficiency
\cite{wells2004electricity}, and
it is expected to play an even more prominent role
with the planned integration of intermittent renewable generation. 
There are  
three levels of DR
automation: manual DR ---where each equipment controller is manually turned
off; semi-automated DR ---where an individual triggers a preprogramed DR
strategy via a centralized control system; and fully automated DR ---where
the DR strategy is initiated by an automated DR device (ADR) on
receipt of an external communications signal~\cite{piette2009design}.
The utilities offering DR programs prefer to install ADRs since the fully
automated system
reduces the operating costs of DR programs
by increasing DR resource reliability and reducing the amount of effort
required from end-use customers~\cite{openADR}. 

ADRs are faulty~\cite{minnesotapower},\cite{silverspringlegacy}, 
and have to be periodically inspected 
and reset  by
sending a maintenance crew. 
Although newer ADRs are equipped with a
two-way communication that allow the utility to remotely observe the ADR
state, the vast majority of the deployed ADRs 
use one-way communication technology and, therefore, their state is not
directly observable.
According to some estimates, failure to identify
non-functioning ADRs can reduce the effectiveness of DR programs by
approximately 20-30\% and  lead to
lost revenues of the order of
\$ 1.7M for a utiliy with 1M customers
and $10$\% participation rate~\cite{silverspringbusinesscase}. Consequently,
identifying malfunctioning ADRs is of immense economic value for the utilities.
The current practice adopted by utilities is to regularly send a maintenance 
team to inspect and possibly repair the ADRs.
In this paper we propose a method that is able to infer the ADR state from
meter readings, and use the estimates to optimally schedule the ADR
maintenance. We show that our proposed method clearly outperforms 
any regular maintenance schedule, without the need of any new investment in 
additional hardware.
The method we propose can also be used by utilities that have invested in 
two-way communication ADRs to verify whether the reported ADR 
state was accurate. 

As a first step towards solving the ADR repair scheduling problem, we
formulate the maintenance problem for a single ADR using only noisy
meter readings. We assume that an ADR can be in 
either in a \emph{functional} or a \emph{non-functional} state. We assume
that, over a given time, the ADR state transitions from a functional to 
a non-functional state with a known probability. Thus, state
transitions form a Markov chain.   The utility 
decides whether to \emph{send} a crew to reset the device or \emph{do
  nothing}. 
This decision would be simple if the true ADR state was directly 
observable without any errors. In most currently deployed ADRs, the state
is not observable;  it can only be inferred from the 
customer's noisy  electricity consumption. Thus, the ADR maintenance
problem can be modeled as a Markov decision process (MDP) with partially 
observable states~(POMDP). 
The POMDP 
framework incorporates the uncertainty associated with any estimation
process into a ``belief'' probability that the ADR is 
functioning,  and provides a methodology for updating this belief 
as more information becomes available, i.e. more meter readings are recorded 
during new DR events.
POMDPs are, in general, very hard optimization problems
\cite{papadimitriou1987complexity},
\cite{madani1999undecidability}.
However,
it is often
possible to compute the optimal policy for POMDPs with small state, action
and observation spaces, or additional structure. 
We show that the optimal policy in our problem is a single threshold policy 
where it is optimal to send a maintenance crew whenever 
the belief probability drops below the threshold value. 
We show that the optimal threshold can be approximated to any degree
of accuracy by solving a single linear program (LP).

Our model for a single ADR maintenance problem falls in the class of random 
failure 
models. Currently existing alternatives to our approach are empirical predictive 
maintenance routines such as the Reliability-Centered Maintenance, and 
deterioration failure models; see~\cite{endrenyi2001present} for details. 
These methodologies are typically used for managing traditional electric assets 
such as generators and transformers. To the best of our knowledge, there is no 
previous work on probabilistic models for the maintenance of a single ADR, let 
alone the maintenance scheduling of multiple ADRs being overseen by a small 
number of repair teams.

We formulate the problem of maintaining a pool of ADRs using a limited
number of maintenance crews as a 
restless bandit (RB) problem~\cite{whittle1988restless}. The RB
problem is a generalization of the  
multiarmed bandit~(MAB) problem~\cite{gittins1979bandit}.
In the MAB problem, the decision maker chooses a set of
``bandits'' to activate based on the current state information, and the state of
the chosen bandits evolves according to a known distribution; however,
the states of all the inactive bandits remain fixed.
 Gittins
\cite{gittins1979bandit} constructed a set of index functions that map the
state of a bandit to a real number, and showed that the optimal solution for
the MAB problem is to select the
bandits in the order determined by the index function. 
In the RB setting, the
states of the inactive bandits also evolve. In the ADR
maintenance setting, 
an 
active ``bandit'' corresponds to an ADR that is being inspected and possibly repaired,
and an inactive ``bandit'' corresponds to an ADR that is not
being inspected; clearly, the state of inactive ADRs continue
to evolve according to its failure
distribution.
Computing optimal policies for the RB problem is hard
\cite{papadimitriou1999complexity}, but a generalization of Gittins' indices can be
used for constructing approximately optimal and very efficiently implementable 
policies for  a number of applications.
Whittle \cite{whittle1988restless}
established conditions under which one can define generalized Gittins'
indices for the RB problem. We will refer to such RB problems as
Whittle-indexable.  
Glazebrook et al. \cite{glazebrook2006some} have established that
machine maintenance models
that are either monotone
or are breakdown/deterioration models,
are Whittle-indexable. 
The ADR repair scheduling problem
is neither  monotone nor a
breakdown/deterioration model; therefore, the results
in~\cite{glazebrook2006some} do not extend to this model.
We establish that the ADR repair scheduling problem is Whittle
indexable,
and show that the Whittle indices
can be computed
by solving a sequence of single ADR maintenance problems.

Finally, we conduct an extensive numerical study where we explore several additional
practical issues such as the impact on performance when the utility and meter
clock are not synchronous, and the impact of uncertain DR load
shed. Our proposed variational Bayes procedure to handle these issues is
of independent interest. 
The results of our
numerical experiments suggest that, for reasonable values of the 
signal-to-noise ratio in the meter readings, the Whittle-index policy is 
within $3.95$\% of the optimal policy.

\section{ADR maintenance problem}
\label{sec:maintenance}
In this section we investigate the maintenance problem for a single ADR. The
solution to this problem will be used to solve a relaxation of the multiple 
ADRs repair scheduling problem in Section~\ref{sec:scheduling}.


For $t \geq 1$, let $s_t \in \cS$ denote
the ADR state just prior to the $t$-th DR event, let $a_t \in \cA$ denote
the action  taken just prior to the $t$-th DR
event, and let $\bx_t$ denote the vector of meter readings recorded during
the $t$-th DR event. 
During DR event $t$, the utility receives a profit $R_{a_t s_t}$. The profit 
matrix $\bR$ is given by 
\begin{equation}
\bR =  [R_{as}]_{a \in \cA, s \in \cS} = \bordermatrix{    
  & \gamma_0 & \gamma_1 \cr
  \alpha_0 & -\theta & \lambda - \theta \cr
  \alpha_1 & \lambda - \theta - c & \lambda - \theta - c }
\end{equation}
For ease of notation, and w.l.o.g., we will assume that $\theta = 0$.
We assume that an ADR that was functioning during the $t$-th event, i.e. $s_t
= \gamma_1$, fails before the $(t+1)$-th DR event, i.e. $s_{t+1} =
\gamma_0$, with probability $p$. Thus, under the do-nothing action
$\alpha_0$, the state transition matrix
\begin{equation}
\bP(\alpha_0) = [P_{ss^\prime}(\alpha_0)]_{ss^\prime \in \cS} = 
\bordermatrix{    
  & \gamma_0 & \gamma_1 \cr
  \gamma_0 & 1 & 0 \cr
  \gamma_1 & p & 1-p }.
\end{equation}
Since the action $\alpha_1$ resets the ADR state to $\gamma_1$ just prior
to the $t$-th DR event, it follows that the state transition matrix 
\begin{equation}
\bP(\alpha_1) = [P_{ss^\prime}(\alpha_1)]_{ss^\prime \in \cS} 
=\bordermatrix{    & \gamma_0 & \gamma_1 \cr
  \gamma_0 & p & 1-p \cr
  \gamma_1 & p & 1-p }.
\end{equation}
We assume that 
the true ADR state can only be determined by 
sending a maintenance crew to inspect it.
However, the utility can use the the vector $\bx$ of metered power consumption 
to infer the state of the ADR.
Let 
\begin{equation}
Q_{sa}(\bx) = \mP(\bx_{t} = \bx \mid a_t=a, s_{t} = s)
\end{equation}
denote the conditional probability of observing the vector of meter readings 
$\bx_{t} = \bx$
during the $t$-th DR event 
when action $a_t = a$, 
and 
the ADR state 
$s_{t} = s$.
The tuple $(\cS,\cA, \cX, \bP, \bR, \bQ)$ is a partially observable Markov
decision process (POMDP) that completely describes the ADR maintenance
problem. Note that we deviate from the  standard definition of POMDP (see,
e.g.~\cite{braziunas2003pomdp}) wherein the observation $\bx_t$ is a function of the
current state $s_t$ and the \emph{previous} action $a_{t-1}$. We do so in
order to keep the dynamics in the problem more transparent.
Figure~\ref{fig:dynamics} describes the causal relationships between
states, actions, rewards, and observations in  the ADR maintenance
problem. 

\begin{figure}
	\centering
	\begin{subfigure}[t]{.2\textwidth}
	  \includegraphics[width=\textwidth]{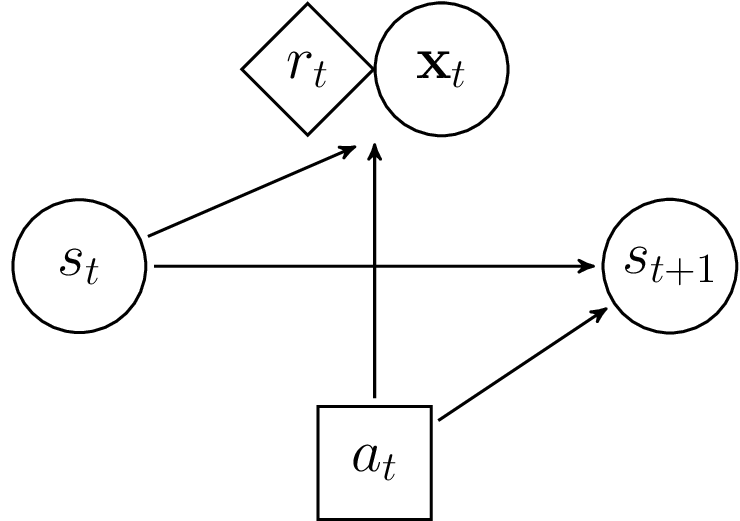}
	  \caption{Partially observable MDP}
	\end{subfigure}	 
	\begin{subfigure}[t]{.2\textwidth}
	  \includegraphics[width=\textwidth]{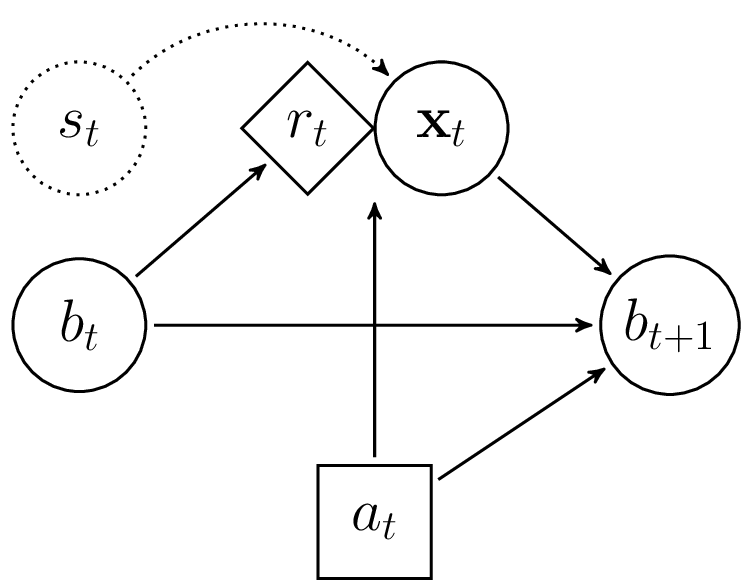}
	  \caption{Belief state MDP}
	\end{subfigure}	
	\caption{ADR maintenance problem causal relationships}
	\label{fig:dynamics}
\end{figure}

It is well-known~\cite{aastrom1965optimal} that a POMDP is equivalent to a 
belief state MDP. 
Let $h_t = \{a_1, \bx_1, a_2, \bx_2, \ldots, a_{t-1}, \bx_{t-1}\}$ denote the
observed history before the $t$-th DR event. Then,
\begin{equation}
b_t = \mP(s_t = \gamma_1 \mid h_t) \in \cB = [0,1]
\end{equation}
denotes the belief probability that the ADR is 
operational during the
$t$-th  DR event. 
The expected profit $\br(b)$ in belief state $b \in \cB$ is given by
\begin{equation}
\br(b) = [r_{\alpha_0}(b), r_{\alpha_1}(b)]^\top = \bR \; [1-b,  b]^\top
= [\lambda b, \lambda - c]^\top.
\end{equation}
Note that the updated belief state $b_{t+1}$ 
after observing 
the meter readings
$\bx_t$ 
is 
$b_{t+1}  =  \mP(s_{t+1} = \gamma_1 \mid  h_t, a_t = a, \bx_t =
  \bx)$.
Using Bayes' rule and the causality structure in
Figure~\ref{fig:dynamics}, one can show that $b_{t+1} = \Gamma_{a_t\bx_t}(b_t)$,
where the map
\begin{equation}
\Gamma_{a\bx}(b)
=
\frac{P_{\gamma_1\gamma_1}(a) Q_{\gamma_1a}(\bx) b +
  P_{\gamma_0\gamma_1}(a) Q_{\gamma_0a}(\bx) (1-b)}
{W_{ab}(\bx)} \label{eq:Gamma-def},
\end{equation} 
and
$
W_{ab}(\bx) = \mP(\bx_t = \bx \mid  h_t,~a_t = a) = Q_{\gamma_1a}(\bx)b+
  Q_{\gamma_0a}(\bx)(1-b)
$. The belief transition \eqref{eq:Gamma-def} is standard for POMDPs; see, e.g.
\cite{monahan1982state} for details.
The tuple $(\cB, \cA, \Gamma,
\br)$ is the belief-state MDP that represents the ADR maintenance problem.

\subsection{Optimal policy}
Our goal is to compute a policy $\pi^\ast:\cB \rightarrow \cA$ that maximizes the
total expected discounted profit
\begin{equation}
V^\pi(b_1) = \sum_{t=1}^\infty \beta^{t-1} r_{a_t^\pi}(b_t),
\end{equation}
where  $\beta \in (0,1)$ is a
discount factor and $a_t^\pi$ denotes the action taken by policy $\pi$ in
belief state $b_t$ prior to the $t$-th DR event. 
It is well known~\cite{sondik1978optimal} that there exists a stationary optimal
policy $\pi^\ast \in \argmax_{\pi} V^\pi(b_1)$, and that the associated optimal
value function $V^{\pi^\ast}(b)$, satisfies the Bellman equation: 
\begin{equation}
  V(b) = \max_{a \in \cA} \bigg\{r_a(b) +
   \beta \mE_\bx \big[V(\Gamma_{a
    \bx}(b))\big]\bigg\}, 
  \label{eq:bellman}
\end{equation}
where $\mE_\bx \left[V(\Gamma_{a \bx}(b))\right] = \int_{\bx \in \cX}
V(\Gamma_{a \bx}(b)) W_{a b}(\bx) \dif{\bx}$ denotes the expected future profit. 
From 
\eqref{eq:bellman} it follows that action
$\alpha_1$ is optimal for belief state $b$ 
if the function
\begin{align}
  f(b) & :=  r_{\alpha_0}(b) - r_{\alpha_1}(b) + \beta \mE_\bx \left[V(\Gamma_{\alpha_0
    \bx}(b)) \right] \nonumber \\ 
& \hspace*{0.2in} - \beta \mE_\bx \left[V(\Gamma_{\alpha_1 \bx}(b)) \right] \leq 0.
\end{align}
We 
use the following result to show that the optimal
policy is a single threshold policy. 
\begin{lemma}
  \label{lm:fconvex}
  The function $f$ is convex in $[0,1]$.
\end{lemma}
\begin{proof}
  Note that $r_{\alpha_0}(b) = \lambda b$ is an affine function of $b$ and
  $r_{\alpha_1}(b) = \lambda - c$ is independent of $b$. It follows that
  $r_{\alpha_0}(b) - r_{\alpha_1}(b)$ is an affine and, therefore,
  convex function of $b$.
  Also, since  $\Gamma_{\alpha_1 \bx}(b) = 1-p$, the
  last term $\mE_\bx \left[V(\Gamma_{\alpha_1 \bx}(b)) 
\right]$ in $f(b)$
  is, in fact, independent of $b$. 

  The value function $V$ of an infinite
  horizon MDP is convex on $[0,1]$ \cite{sondik1978optimal}. Then,
  the
  perspective function $g(v,u) = uV(v/u)$ is convex on $\{(v,u) : v/u
  \in  [0,1], v \in \reals_+\}$~\cite{boyd2009convex}.
  Since 
  $W_{\alpha_0b}(\bx) = Q_{\gamma_1\alpha_0}(\bx)b+
  Q_{\gamma_0\alpha_0}(\bx)(1-b)$
  and
	\begin{equation}
  \Gamma_{\alpha_0 \bx}(b) = \frac{(1-p)Q_{\gamma_1\alpha_0}(\bx)
    b}{W_{\alpha_0 \bx}(b)}, 
	\end{equation}
  it follows that 
  \begin{eqnarray}
  \lefteqn{g_{\bx}(b) := V(\Gamma_{\alpha_0 \bx}(b)) W_{\alpha_0 \bx}(b)} \\
  & = & 
  g\Big((1-p)Q_{\gamma_1\alpha_0}(\bx) b, \; Q_{\gamma_1\alpha_0}(\bx)b+
  Q_{\gamma_0\alpha_0}(\bx)(1-b)\Big) \nonumber
\end{eqnarray}
is a convex function of $b$ for each $\bx \in \cX$.
  Hence, $\mE_\bx \left[V(\Gamma_{\alpha_0
      \bx}(b))\right] = \int_{\bx \in \cX}
g_{\bx}(b) \dif{\bx}$ 
  is a convex function of
  $b$.
 \end{proof}

\begin{theorem}
  \label{thm:threshold-opt}
The optimal policy $\pi^\ast$ for the ADR maintenance problem is a
threshold policy, i.e. there exists $b^\ast \in \reals$ such that 
\begin{equation}
  a_t^{\pi^\ast} = \left\{\begin{array}{rl}
      \alpha_1 & \text{if } b_t \leq b^\ast \\
      \alpha_0 & \text{if } b_t > b^\ast. \end{array}\right.
  \label{eq:policy}
\end{equation}
\end{theorem}
\begin{proof}
  From Lemma \ref{lm:fconvex}, we have that $f(b)$ is
  convex. Consequently, the set $I = \{b: f(b) \leq 0\}$ of belief states
  for  which action $\alpha_1$ is optimal is a closed, possibly empty,
  interval. Suppose $I$ is empty. Then the optimal
  policy is of the form (\ref{eq:policy}) with $b^\ast<0$.

  Next, suppose $I$ is non-empty. To establish the result, it suffices to show that $0 \in
  I$. Suppose not, i.e. $\alpha_0$ is strictly optimal for $b = 0$. Then,
  \begin{equation}
  V(0) = r_{\alpha_0}(0) + \beta \mE_{\bx}[V(\Gamma_{\alpha_0\bx}(0))] = 0
  + \beta V(0), 
	\end{equation}
  where we use fact that $\Gamma_{\alpha_0\bx}(0) = 0$. It follows
  that $V(0) = 0$. Moreover,
  \begin{equation}
  0 = V(0) > r_{\alpha_1}(0) + \beta \mE[V(\Gamma_{\alpha_1\bx}(0))] =
  \lambda - c + \beta V(1-p).
  \end{equation}
  Then, for any belief state $b$,
  \begin{equation}
  r_{\alpha_1}(b) + \beta \mE[V(\Gamma_{\alpha_1\bx}(b))] =
  \lambda - c + \beta V(1-p) < 0.
  \end{equation}
  Since the payoff from action $\alpha_0$ is non-negative in
  any belief state, we have that $V(b) \geq 0$ for all $b \in \cB$.
  It follows that $V(b) = r_{\alpha_0}(b) +
  \beta\mE_{\bx}[V(\Gamma_{\alpha_0\bx}(b))]$, and $\alpha_0$ is the
  unique optimal action 
  for all belief states. Thus, we have established that the
  interval $I$ is empty; a contradiction.
	Hence, $0 \in I$ and the optimal policy is of the form
(\ref{eq:policy}) with $b^\ast \in [0,1]$.
\end{proof}
We conclude this section 
with the following corollary.
\begin{corollary}
  \label{cor:f-monotone}
  The
  level sets $\{b: f(b) \leq -\mu\}$ of the function $f$
  are all of the form $[0,\xi]$, 
where we use the convention that $[0, \xi] = \emptyset$ if $\xi 
< 0$.
\end{corollary}

\begin{proof}
Consider an ADR maintenance problem where 
the reward associated with the do-nothing action $\alpha_0$ is increased by  an
amount $\mu$. Then the 
interval over which it is optimal to take action $\alpha_1$ is given by 
$\{b: f(b)  + \mu \leq 0 \}$.
From the proof of Theorem \ref{thm:threshold-opt}, it follows that the set
$\{b: f(b)  \leq -\mu\}$ is of the  form $[0, 
\xi]$. 
\end{proof}
We use Corollary~\ref{cor:f-monotone} in Section \ref{sec:scheduling} to establish that
the ADR repair scheduling problem 
is Whittle-indexable~\cite{glazebrook2006some} and, therefore, that
there exists a well-defined heuristic policy for approximately solving it.

\subsection{Approximating the optimal threshold}
In this section, we describe a numerical scheme to approximate the optimal
threshold $b^\ast$.  
We discretize the space $\cB = [0,1]$ into $n+1$ equally spaced points
$\hat{\cB} = \big\{\sfrac{k}{n} : k = 0, \dots, n\big\}$, and round up
$b \in \big(\sfrac{(k-1)}{n}, 
  \sfrac{k}{n}\big]$ to the point $\sfrac{k}{n}$.  
Formally, we consider the MDP with state space
$\hat{\cB}$, action space $\cA$, reward function $\br\big|_{\hat{\cB}}$
and state transition $\hat{\Gamma}_{a\bx}(k) := \lceil
\Gamma_{a\bx}(\sfrac{k}{n})\rceil$. 
Let $\hat{\bV} = (\hat{V}(0), \ldots, \hat{V}(k), \ldots, \hat{V}(n))^\top$ 
denote the value
function vector for the $n+1$ states $\sfrac{k}{n} \in \hat{\cB}$. Then
\begin{equation}
  \hat{V}(k)  =  \max \bigg\{ 
    \begin{array}{l}
      r_{\alpha_0}(\sfrac{k}{n}) +
      \beta \mE_{\bx}[\hat{V}(\hat{\Gamma}_{\alpha_0\bx}(k)],\\
      r_{\alpha_1} ( \sfrac{k}{n} ) + \beta
      \mE_{\bx}[\hat{V}(\hat{\Gamma}_{\alpha_1\bx}(k)]
    \end{array} \bigg\}.
 \label{eq:valuefnapprox}
\end{equation}
It is well-known~\cite{puterman2009markov} that the vector $\hat{\bV}$ can be 
computed by solving
the LP: 
\begin{equation}
  \label{eq:threshold-lp}
  \begin{array}[t]{rll}
    \min\limits_{\bV} 	& \sum\limits_{k = 0}^{n} V(k) \\
    \text{s.t} 	& V(k)  \geq r_{\alpha_0}(\sfrac{k}{n}) + 
    \beta  \mE_{\bx}[V(\hat{\Gamma}_{\alpha_0\bx}(k))], &
    \forall k,\\ 
    & V(k) \geq r_{\alpha_1}(\sfrac{k}{n}) + 
      \beta V(\lceil(1-p)n\rceil), &   \forall k,
    \end{array}
\end{equation}
where we use the fact that $\hat{\Gamma}_{\alpha_1\bx}(k) = \lceil
(1-p)n \rceil $. LP~\eqref{eq:threshold-lp} consists of $n$
variables and $2n$ 
constraints and, therefore, can be solved very fast, provided  the
conditional expectations $\mE_{\bx}[V(\hat{\Gamma}_{\alpha_0\bx}(k))]$ can be
computed efficiently. In Section
\ref{sec:numexperiments}, we show how to efficiently approximate the
conditional expectations in (\ref{eq:threshold-lp}) using low discrepancy
sequences. Given the approximate value function $\hat{\bV}$, we approximate
the threshold
$b^\ast \approx \frac{1}{n}\max\{k : \hat{V}(k) =
r_{\alpha_1}(\sfrac{k}{n}) + \beta \hat{V}(\lceil(1-p)n\rceil) \}$. 

\section{ADR repair scheduling problem}
\label{sec:scheduling}
Suppose there are $M$ maintenance crews available to maintain $D (\gg M)$
ADRs. Before each DR event $t$, the utility must decide which (if any)
ADRs to inspect and reset.   
We formulate the ADR repair scheduling problem as a restless bandit (RB)
problem~\cite{whittle1988restless} where each of the ADR devices is a
bandit. Following the notation of 
bandit problems, we will call an ADR active if it is being
inspected, and inactive otherwise.

For RB problems, Whittle~\cite{whittle1988restless} proposed a possibly 
suboptimal policy 
based on the 
Lagrangian relaxation 
for the dynamic program.
The Lagrangian relaxation decouples the RB problem into
a collection of so called subsidy-$\mu$ problems.
For each bandit, the subsidy-$\mu$ problem is an ADR maintenance problem in which the
reward under the do-nothing action $\alpha_0$ is increased
by an amount $\mu \geq 0$.

Let $\cB_0(\mu)$ denote the set of states for which the do-nothing action 
$\alpha_0$
is optimal for a subsidy level $\mu$. An RB problem is said to be
\emph{indexable} if $\cB_0(\mu) \subseteq \cB_0(\zeta)$ whenever $\mu \leq
\zeta$.
For indexable RB problems, the Whittle index $\mu_\kappa$ of bandit $\kappa$ in
state $b_\kappa$ is defined as the minimum subsidy $\mu \geq 0$ which makes the 
passive
action $\alpha_0$ optimal at $b_\kappa$. 
Let $\cK = \{\kappa_{(\ell)}: \ell = 1, \ldots, K\}$ denote the indices of
bandits with strictly positive Whittle index arranged in decreasing order
of the index, i.e. $\mu_{\kappa_{(\ell)}} \geq \mu_{\kappa_{(\iota)}}$ whenever 
$\ell \leq
\iota$. The Whittle-index policy specifies that the set of bandits that are
activated at time $t$ is given by  $\cK_1 = \{\kappa_{(\ell)}: 1 \leq \ell \leq
\min \{K, M\}\}$, i.e. at most $M$ bandits with the largest strictly
positive Whittle index are activated. 
The next theorem ensures that the Whittle indices are well defined for our
problem. 
\begin{theorem}
  \label{thm:indexable}
  The ADR repair scheduling problem is \textit{indexable}.
\end{theorem}
\begin{proof}
  The Bellman equation for the subsidy-$\mu$ problem corresponding to the
  ADR scheduling problem is given by 
  \begin{equation}
    V_\mu(b) = \max_{a \in \cA} \left\{r_a(b, \mu) + \beta \mE_\bx
      [V_\mu(\Gamma_{a \bx}(b))]\right\}, 
    \label{eq:subsidy} 
  \end{equation}
  where $\br(b, \mu) = \br(b) + (\mu, 0)^\top$.
  The set of states for which $\alpha_0$ is optimal
  is of the form $\cB_0(\mu) = \{b \in [0,1] : b > b^\ast(\mu)\}$ where
  $b^\ast(\mu)$ denotes the optimal threshold corresponding to the 
  subsidy-$\mu$ problem~\eqref{eq:subsidy}. Hence,
  in order to establish that the ADR repair scheduling problem is
  indexable, we need to show that the threshold $b^\ast(\mu)$ is non-increasing
  in $\mu$. 
  From Corollary~\ref{cor:f-monotone}, it
  follows that $[0,b^\ast(\mu)] = 
  \{b: f(b)  \leq -\mu \}$. 
  Thus, $b^\ast(\mu)$ is 
  non-increasing in $\mu$.
\end{proof}
Since the
Whittle index 
$\mu^\ast(b) = \inf\{\mu \geq 0: b \in \cB_0(\mu)\} = \inf\{\mu \geq 0: b >
b^\ast(\mu) 
\}$, it is clear that $\bar{\mu} = \inf\{\mu \geq 0: b^\ast(\mu) < 0\}$ is  
an upper bound for $\mu^\ast(b)$. 
Therefore, the Whittle
index for any $b \in B$ can be computed to within an accuracy
$\epsilon$ using a binary search 
by solving at most $\cO(\log_2 (\frac{\bar{\mu}}{\epsilon}))$ LPs of
the form (\ref{eq:threshold-lp}). Thus, in practice, one needs to solve at most
$\cO(n\log_2(\frac{\bar{\mu}}{\epsilon}))$ LPs.
On the other hand, we can also
compute Whittle indices to within $\epsilon$ accuracy by finding the optimal
threshold $b^\ast(\mu)$ for all $\mu \in \cM = \{k\epsilon: 0 \leq k \leq
\lceil \frac{\bar{\mu}}{\epsilon} \rceil \}$. Thus, the overall complexity
of computing the Whittle indices is at most
$\cO(\min\{n\log_2(\frac{\bar{\mu}}{\epsilon}), \lceil
\frac{\bar{\mu}}{\epsilon} \rceil\})$. 
In our numerical experiments, we calculated
a bound on $\bar{\mu}$ using a doubling strategy, and
used the binary search approach to compute the Whittle indices.



\section{Numerical implementation}
\label{sec:numexperiments}
Recall that as long as we can compute the conditional expectation
$\mE_{\bx}[V(\hat{\Gamma}_{\alpha_0\bx}(k))]$ efficiently, we can
solve LP~\eqref{eq:threshold-lp} and, consequently, implement the
Whittle-index policy.
In this section, 
we 
discuss how to efficiently compute the required conditional expectation in
two situations that arise in the DR context:
when
the DR load-shed is random, and 
when the clocks at the utility and at the meter are 
not synchronized. We show that these issues can have a significant impact on
the performance of the Whittle-index policy, and propose a variational
Bayes procedure to address them.
Finally, we evaluate the performance of our proposed policies
with respect to periodic review policies, and with respect to policies that
have full information of the state of the ADRs.
\subsection{Deterministic load-shed and synchronized clocks}  
\label{sec:synchronized-readings}
Let $\bx = (x_1, \ldots, x_m)$ denote the sampled power consumption over a
DR event of length $m$ periods. 
For $i = 1, \ldots, m$, we assume that the power consumption $x_i$ is given by 
\begin{equation}
x_i = \left\{\begin{array}{ll}
    y_i - r + \tilde{\varepsilon}_i & \text{if the ADR is operational} \\
    y_i + \tilde{\varepsilon}_i & \text{otherwise}, \end{array}\right. \;
\end{equation}
where $y_i$ denotes the estimated power consumption on a non-DR day, $r$ 
denotes the load-shed mandated by the utility, and the estimation residual
$\tilde{\varepsilon}_i$ 
is assumed to be
independent and identically distributed (IID) according to a normal
distribution  with mean $0$ and variance $\sigma^2$.
As will become clear
below, we can work with any specification for the
residuals $\varepsilon_i$ as long as one is able to simulate
samples from the distribution.

We approximate the conditional expectation
$\mE_\bx[V(\hat{\Gamma}_{\alpha_0\bx}(k))]$  in the 
LP~\eqref{eq:threshold-lp} by the finite weighted sum
\begin{equation}
\mE_\bx[V(\hat{\Gamma}_{\alpha_0\bx}(k))]
\approx \frac{1}{N} \sum_{j=1}^{N} V(\hat{\Gamma}_{\alpha_0\bx^{j}}(k)),
\label{eq:expectation}
\end{equation}
where the samples
$\{\bx^{j}: j = 1, \ldots, N\}$ are generated IID from the distribution
$W_{\alpha_0, b}(\bx)$, $b = \sfrac{k}{n}$, i.e. 
\begin{equation}
  \bx^j \sim \left\{ \begin{array}{ll}
      \cN(\by, \sigma^2 \bI) & \text{w.p. } 1-b, \\
      \cN(\by - r\ones, \sigma^2 \bI) & \text{w.p. } b.
    \end{array}\right.
\end{equation}
To efficiently sample from the multivariate Normal distributions, we use
low-discrepancy sequences in the unit $m$-dimensional
hypercube~\cite{sobol1998quasi}. 
Let $\bfomega^1,$ $\ldots, \bfomega^{N} \in [0,1]^m$ be the first $N$
elements of a low-discrepancy sequence. We
define $\bz^j \in \reals^m$ by setting $z_i^j = \varPhi^{-1}(w_i^j)$. Then 
$\{\bz^{j}: j = 1, \ldots, N\}$ are IID
samples from an $m$-dimensional standard Normal random
variable~\cite{devroye1986sample}. We generate meter readings sample $\bx^{j}$ 
as follows:
\begin{equation}
  \bx^j = \left\{ \begin{array}{ll}
      \by + \sigma \bz^j & \text{w.p. } 1-b, \\
      \by  - r\ones + \sigma \bz^j & \text{w.p. } b.
    \end{array}\right.
\end{equation}
Let $Q_0(\bx)$ (resp. $Q_1(\bx)$) denote the probability of observing a
vector of meter readings $\bx$ under an functional (resp. non-functional) ADR. 
Then, $Q_{\gamma_0 \alpha_0 } (\bx) = Q_{0} (\bx)$, and
$Q_{\gamma_0 \alpha_1} (\bx) = Q_{\gamma_1 \alpha_0} (\bx) = Q_{\gamma_1
  \alpha_1} (\bx) = Q_{1} (\bx)$, 
where
\begin{equation}
  Q_{0} (\bx) = \prod_{i=1}^m \varphi\left(\frac{x_i -
      y_i}{\sigma}\right),~
  Q_{1} (\bx) = \prod_{i=1}^m \varphi\left(\frac{x_i - y_i +
      r}{\sigma}\right). 
\end{equation}
Given $Q_0(\bx^j)$ and $Q_1(\bx^j)$, 
we compute $\hat{\Gamma}_{\alpha_0, \bx^j}(k) = \lceil \Gamma_{\alpha_0, \bx^j}(b) \rceil$ using
(\ref{eq:Gamma-def}). 

\subsection{Random load-shed  and unsynchonized clocks}   
\label{sec:delayed-readings}
Here, we assume that the distribution of the load-shed is
known --however, the exact magnitude is unknown-- 
and that the clock at the utility and in the meter 
can be 
mismatched up to $d$ time units. 
We 
use the utility's clock time as the reference time for all
computations, and  assume that 
meter readings are
assigned to time 
instants using the meter clock; 
therefore, an observation assigned to the time
instant~$i$ could, in fact,  
correspond to an actual instant in the set $\{i-d, \ldots, i+d\}$.  
Suppose a DR event takes place during the period $\{d+1,
\ldots, d+m\}$. Then,
\begin{equation}
  \bx \sim \left\{ \begin{array}{ll}
      \cN(\by - r\ones_\delta, \sigma^2 \bI) & \text{for some } r,
        \; \delta \in 
      \{-d, \ldots, d\}\\
      & \hspace{0.1in} \text{when the ADR operational},\\
      \cN(\by, \sigma^2 \bI) & \text{when the ADR is not operational,} \\
    \end{array}\right.
\end{equation}
where $\ones_{\delta} \in \reals^{m+2d}$ denotes a vector with the
components $i \in I_\delta = \{1 + d + \delta, \ldots, m + d + \delta\}$
equal to $1$, and all other components equal to zero. 

Let $\bfomega^1, \ldots, \bfomega^{N}$ be the first $N$ elements of a
low-discrepancy sequence in the unit $(m+2d)$-dimensional hypercube and $z_i^j =
\varPhi^{-1}(w_i^j)$. Then, the $j$-th sample from the distribution
$W_{\alpha_0, b}(\bx^j)$ is given by 
\begin{equation}
  \bx^j = \left\{ 
    \begin{array}{ll}
      \by + \sigma \bz^j & \text{w.p. } 1-b, \\
      \by  - r\ones_\delta + \sigma \bz^j & \text{w.p. }
      b\rho(r, \delta),
    \end{array}\right.
\end{equation}
where $\rho$ denotes the joint probability distribution
function of the 
load-shed and the
clock mismatch.
In this case, the observation probabilities are
\begin{align}
  Q_{0} (\bx) &= \prod_{i=1}^{m+2d} \varphi\left(\frac{x_i -
      y_i}{\sigma}\right), \\
  Q_{1} (\bx) &= 
  \mE_{\bx} \bigg[
  \prod_{i \in
    I_\delta} \varphi\left(\frac{x_i - y_i + 
      r}{\sigma}\right)  \prod_{i \notin I_\delta} \varphi\left(\frac{x_i
      - y_i}{\sigma}\right)\bigg], \nonumber
\end{align}
where the expectation is with respect to the posterior distribution
of the mismatch $\delta$ and the load-shed $r$. 
In the next section, we show how to use the variational Bayes procedure
 to approximate the posterior expectation. 
Given $Q_1(\bx^j)$ and $Q_0(\bx^j)$,
we approximate $\mE_\bx[V(\hat{\Gamma}_{\alpha_0\bx}(k))]$ 
using~\eqref{eq:expectation} and \eqref{eq:Gamma-def}.

\subsection{Belief state update using variational Bayes}
\label{sec:beliefupdate}
Consider the following
hierarchical Bayesian model 
\begin{equation}
\begin{array}{l}
  \bx = \by -r\ones_\delta + \bfvarepsilon, \hspace{30pt} r \sim \cN(\nu_o,
  1/\eta_0),\\
  \delta \sim 
  \text{unif}\{-d,\ldots,d\},
  \hspace{13pt}   \bfvarepsilon \sim \cN(\bo, \sigma^2 \bI),
\end{array}
\end{equation}
where $\nu_0$ is the expected load-shed during the DR
event, and $\eta_0^{-\sfrac{1}{2}}$ is its standard deviation.
The posterior distribution $\rho(r, \delta|\bx,\cH)$
is proportional to 
$ \mP(\bx | r, \delta, \sigma) \mP(r | \nu_0, \eta_0) 
\mP(\delta)$, where
\begin{equation}
\mP(\bx | \delta, r, \sigma) =
\prod_{i \in I_\delta} \varphi\left(\frac{x_i - y_i +
    r}{\sigma}\right) \prod_{i \notin I_\delta} \varphi\left(\frac{x_i -
    y_i}{\sigma}\right).
\end{equation}
Since this posterior distribution is neither in closed form,
nor is it easy to sample from, 
we use the variational inference method
\cite{bishop2006pattern} 
and approximate the posterior distribution 
by the
product distribution $\Delta(\delta)g(r)$. It is easy to establish that the
product distribution that minimizes the Kullback-Leibler distance from the
joint posterior distribution is of the form
\begin{eqnarray}
  \Delta(\delta) 
  &\propto&  \exp\left\{\frac{\nu}{\sigma^2}
    \ones_\delta^\top(\by-\bx) -   
    \frac{\nu^2 + \eta^{-1}}{2\sigma^2}\norm{\ones_\delta}^2_2\right\}\\
  g(r) &=& \cN(\nu, 1/\eta),
\end{eqnarray}
where 
$\eta = \eta_0 + \frac{1}{\sigma^2}\sum_{\delta=-d}^{d} \Delta(\delta)
\norm{\ones_\delta}^2_2$ and
$ \nu = \eta^{-1}\Big( \nu_0\eta_0 + \frac{1}{\sigma^2}
\sum_{\delta=-d}^{d} \Delta(\delta) \ones_\delta^\top (\by -
\bx)\Big)$. 
Note the circular dependence of the parameters in the posterior
distributions. We use an iterative procedure that alternates
between 
$\Delta(\delta)$, and the precision $\eta$ and mean $\nu$ of the
normal distribution for $r$. We terminate the procedure whenever
the relative change in the parameters is small. We approximate the
probability $Q_{1}(\bx)$ using the posterior distributions
$\Delta(\delta)$ and $g(r)$ as follows: 
\begin{align}
Q_{1}(\bx) & \approx \sum_{\ell=-L}^{L} \sum_{\delta = -d}^{d} \prod_{i
  \in I_\delta} \varphi\left(\frac{x_i - y_i + 
    r_\ell}{\sigma}\right) \\ \nonumber
& \hspace{57pt} \prod_{i \notin I_\delta} \varphi\left(\frac{x_i -
    y_i}{\sigma}\right) \Delta(\delta) g(r_\ell), 
\end{align} 
where $r_\ell$ is $\ell$ standard deviations $\eta^{-\sfrac{1}{2}}$ above
the mean $\nu$. The posterior distribution in the case where either the meter readings
are perfectly synchronized or the load-shed magnitude is deterministic can
be computed as a special case of this procedure. 


\subsection{Problem parameters and available information}
\label{sec:cases}
We consider
the ADR maintenance problem
in
the following four settings: 
\begin{enumerate}[C{a}se (a)]
\item 
  Clocks
  synchronized and load-shed deterministic. \label{case:known-synch}
\item
  Clocks possibly mismatched with $d =
  2$, but load-shed deterministic.
  \label{case:known-delay} 
\item Clocks synchronized, but 
  load-shed distributed $\cN(\nu_0,\frac{1}{\eta_0})$
  \label{case:unknown-synch} 
\item Clocks possibly mismatched with $d =
  2$, and load-shed distributed 
  $\cN(\nu_0,\frac{1}{\eta_0})$.  
\label{case:unknown-delay} 
\end{enumerate}

The parameter values for the numerical experiments were set as follows:
\begin{enumerate}[(i)]
\item Observations in an hour-long DR event  $m = 10$
\item Probability of ADR failure  $p=0.05$
\item Expected DR savings for utility $\lambda = 1$
\item Cost of repair $c = 3 \lambda = 3$
\item Discount factor $\beta = 0.9$
\item Load-shed signal-to-noise ratio $\text{SNR} = 20
  \log_{10}(\sfrac{\nu_0}{\sigma}) = \{-5,0,5\}$dB, i.e. standard
  deviation $\sigma$ of the meter noise 
  $\approx\{1.78,1,0.56\}$ times the  mean load-shed 
  $\nu_0$
\item Load-shed standard deviation 
  $\eta_0^{-\sfrac{1}{2}}
  = 0.1\sigma$ 
\end{enumerate}

\subsection{Discretization size $n$ and sample size $N$}
In Table \ref{tab:LP}, we report the optimal threshold $b^\ast$, the approximate
optimal value function $V(b)$ at belief state $b=1$,  and the
elapsed time in seconds to approximate the expectation
$\mE_\bx[V(\hat{\Gamma}_{\alpha_0\bx}(k))]$ and solve 
LP~\eqref{eq:threshold-lp}, 
as a function
of the number $n$ of points used to discretize the interval $[0,1]$ and
the number of samples $N$.
From the results, it is clear that the number of samples $N$ does not have
a significant 
impact on the value function 
$V(1)$. 
We interrupted the computation of the optimal threshold and the optimal
value for Case~\eqref{case:unknown-delay} for the largest-sized
approximations, i.e. $N=500$K. The solution time was too large to
be of practical use. In Table~\ref{tab:LP}, we report this situation with
an horizontal line. 
For the rest of the experiments in this section, we used $n=100$ and $N=5$K.  

\begin{table}
\centering
\begin{tabular}{rr|rrrr|c}
$n$ & $N$ & Case~\eqref{case:known-synch} & Case~\eqref{case:known-delay}
& Case~\eqref{case:unknown-synch} & Case~\eqref{case:unknown-delay} \\ \hline
100 & 5K & 0.160 & 0.150 & 0.170 & 0.150 & \multirow{4}{*}{$b^\ast$} \\
100 & 500K & 0.160 & 0.150 & 0.160 & ----- \\
1K & 5K & 0.168 & 0.154 & 0.171 & 0.159 \\
1K & 500K & 0.163 & 0.154 & 0.166 & ----- \\ \hline
100 & 5K & 8.374 & 8.558 & 8.498 & 8.713 & \multirow{4}{*}{$V(1)$} \\
100 & 500K & 8.309 & 8.550 & 8.448 & ----- \\
1K & 5K & 8.292 & 8.477 & 8.423 & 8.648 \\
1K & 500K & 8.275 & 8.515 & 8.412 & ----- \\ \hline
100 & 5K & 1 & 2 & 8 & 125 & \multirow{4}{*}{time(s)} \\
100 & 500K & 25 & 132 & 763 & ----- \\
1K & 5K & 86 & 102 & 159 & 1291 \\
1K & 500K & 334 & 1373 & 8205 & ----- 
\end{tabular}
\caption{Optimal threshold $b^\ast$, optimal value function $V(b)$ at
  belief state $b=1$, and time in seconds to approximate
  $\mE_\bx[V(\hat{\Gamma}_{\alpha_0\bx}(k))]$ and solve
  LP~\eqref{eq:threshold-lp}, 
  as a function of   $n$ and
  $N$.} 
\label{tab:LP}
\end{table}

\subsection{Comparison with periodic review policies}
In this section we report the 
performance of the proposed threshold policy with respect to
the current practice 
of periodically inspecting the ADRs.
Suppose the periodic review interval
is $q$ DR events. Then, the associated value function $U(q)$ satisfies
the recursion 
\begin{equation}
U(q) = -c + \lambda \sum_{j=0}^{q-1} \beta^j(1-p)^j + \beta^{q} U(q),
\end{equation}
i.e. $U(q) = \frac{1}{1-\beta^q} \left(\frac{\lambda(1 - \beta^q (1-p)^q)}{1 -
      \beta(1-p)} - c \right)$.
The value function $U(q)$ is maximized at $q^\ast=18$, where $U(q^\ast) = 4.10$.
In contrast, the POMDP value function is at least $V(0) = \{5.14, \; 5.37, \; 5.51\}$
for $\text{SNR} = \{-5, 0, 5\}$dB, i.e. a relative improvement over 
$U(q^\ast)$
of at least $\{25, \; 31, \; 34\}$\%.
These results clearly show that the POMDP-based method significantly outperforms any
regular maintenance schedule. 

\subsection{ADR repair scheduling problem}


In this section we report the numerical results for an ADR repair
scheduling problem with 
$D=100$ ADRs 
and $M=5$ repair crews.
We considered two variants of the ADR repair scheduling problem: one
where the repair costs for all the ADRs were identical and set
equal to $3\lambda$, 
and another where repair cost for each ADR were
sampled uniformly from the interval~$(0,6.5]\lambda$. 
For $c > 6.5 \lambda$,
the utility is better off not repairing the ADR at all.

When the states of the ADRs are fully observable and the repair costs are 
identical, the associated multi-dimensional DP can be
reduced to a one-dimensional DP with state given by the
number of non-working devices.  Thus, one can easily compute 
the optimal policy using the value iteration algorithm. In this case, we
compare the performance of the Whittle-index policy for the POMDP with
this optimal 
policy. When the states of the ADRs are fully observable but the repair costs 
are not identical, it is not possible to compute the optimal
policy. In this setting, we compare the Whittle-index policy for the
POMDP with the Whittle-index policy for the full
information MDP.
Given that the partially observable Whittle-index policy can 
infer the state of an ADR only \emph{after} observing the meter readings
during a DR event, a fairer comparison would be 
against the  optimal policy of the  MDP where the state of the ADR is observable 
only after a DR event. We call this problem, the \emph{slow
  information} MDP.

We computed the value function of the Whittle-index policies
using simulation. 
We simulated the performance of the policies over a time horizon of $T=44$ DR 
events with a discout factor $\beta=0.9$. This implies that the simulated 
$T$-horizon value function is within 1\% of the infinite horizon value function.
The
results are averaged over $100$ runs.
The column marked ``Whittle'' in Table \ref{tab:identical} reports the relative
error of the full information Whittle-index policy 
with respect to the optimal full
and slow information policy when all the repair costs are identical.   The next
four columns report 
the performance of the partial information Whittle-index policy for the
four different cases listed in
Section~\ref{sec:cases}. Since the partial information policy does not have
access to the 
state, its performance depends on the load-shed SNR. The row marked by SNR 
$[-5,5]$
reports the performance of the policy when the SNR of each ADR is sampled
uniformly from the interval $[-5,5]$; all other rows report the
performance when the SNR for all ADRs was set equal to the value
corresponding to that row. 

The full
information Whittle-index policy is close to the optimal policy both in the full
information and the slow information case.
For
very low noise levels, i.e. SNR=$5$dB, the performance of partial
information Whittle-index policy is no more than $4.63$\% (resp. $1.57$\%)
worse than that of the full (resp. slow) information
optimal policy. On the other hand, for very high noise levels,
i.e. SNR=$-5$dB, the partial information 
Whittle-index policy could be as bad as $10.75$\% (resp. $7.89$\%)
suboptimal with respect to the full (resp. slow) information optimal policy. For
reasonable noise levels, i.e. SNR=$0$dB for all ADRs, or 
randomly drawn from $[-5, 5]$dB, the performance of the partial
information Whittle index 
policy is within $6.94$\% (resp. $3.95$\%) 
of the full (resp. slow)
information optimal policy.  The results in this table suggest that
the four cases in decreasing order of the sub-optimality of the Whittle policy
are: 
\eqref{case:unknown-delay}, \eqref{case:known-delay},
\eqref{case:unknown-synch}, \eqref{case:known-synch}. Thus, 
it appears that 
the uncertainty in clock mismatch leads to an increased loss in
performance as compared to the uncertainty in the load-shed. 

Table \ref{tab:different} shows the relative error of the partial
information Whittle-index policy with respect to the full and slow
information Whittle-index policy when the repair costs are not identical.
Assuming that the full information Whittle-index policy is close to the
optimal policy also in this case, the results in this table are similar to
those 
observed in the case of identical repair costs. 

\begin{table}
\centering
\begin{tabular}{c|rrrrr|c}
\multirow{2}{*}{SNR}&\multicolumn{5}{c|}{err (\%)}\\
& Whittle & Case \eqref{case:known-synch} & Case
\eqref{case:known-delay} & Case \eqref{case:unknown-synch} & Case
\eqref{case:unknown-delay} \\ \hline 
5 & 1.28 & 4.63 & 4.62 & 4.60 & 4.61 & \multirow{4}{*}{Full} \\
0 & 1.28 & 5.38 & 5.92 & 5.77 & 6.56 &  \\
-5 & 1.28 & 8.48 & 9.63 & 9.23 & 10.75 &  \\
$\text{[-5,5]}$ & 1.28 & 5.76 & 6.3 & 6.14 & 6.94 &  \\ \hline \hline
5 & 1.66 & 1.57 & 1.56 & 1.54 & 1.55 & \multirow{4}{*}{Slow} \\
0 & 1.66 & 2.35 & 2.9 & 2.75 & 3.57 &  \\
-5 & 1.66 & 5.55 & 6.73 & 6.32 & 7.89 &  \\
$\text{[-5,5]}$ & 1.66 & 2.74 & 3.3 & 3.13 & 3.95 &
\end{tabular}
\caption{Identical repair costs. Average error of the full, slow, and partial
  information Whittle-index policies w.r.t. the optimal full and
  slow information policies. See Section~\ref{sec:cases}
  for the details on partial information 
  Cases~\eqref{case:known-synch}-\eqref{case:unknown-delay}.} 
\label{tab:identical}
\end{table}

\begin{table}
\centering
\begin{tabular}{c|rrrr|c}
\multirow{2}{*}{SNR}&\multicolumn{4}{c|}{err (\%)}\\
& Case~\eqref{case:known-synch} & Case~\eqref{case:known-delay} & Case~\eqref{case:unknown-synch} & Case~\eqref{case:unknown-delay} \\ \hline
5 & 3.28 & 3.32 & 3.32 & 3.39 & \multirow{4}{*}{Full} \\
0 & 4.26 & 4.87 & 4.7 & 5.45 &  \\
-5 & 7.15 & 8.04 & 7.73 & 8.93 &  \\
$\text{[-5,5]}$ & 4.61 & 5.15 & 4.96 & 5.66 &  \\ \hline \hline
5 & -0.15 & -0.11 & -0.1 & -0.03 & \multirow{4}{*}{Slow} \\
0 & 0.87 & 1.5 & 1.32 & 2.1 &  \\
-5 & 3.86 & 4.78 & 4.46 & 5.7 &  \\
$\text{[-5,5]}$ & 1.23 & 1.78 & 1.59 & 2.32 &
\end{tabular}
\caption{Non-identical repair costs. Average error of the partial information
  Whittle-index policy with respect to the full and slow information
  Whittle-index policies. See Section~\ref{sec:cases}
  for the details on partial information Cases~\eqref{case:known-synch}-\eqref{case:unknown-delay}.} 
\label{tab:different}
\end{table}

\section{Conclusions}
In this paper, we formulated and solved the ADR repair scheduling
problem where the goal is to maintain
$D$ ADRs using at most $M (\ll D)$ maintenance crews and the ADR state is
only partially observable via noisy meter readings. We formulated this
problem as a 
restless bandit problem. We showed that the ADR repair is Whittle-indexable,
and therefore, one can very efficiently compute a good heuristic policy by suitably
decomposing the $D$-ADR repair scheduling problem into $D$ single-ADR 
maintenance
problems. We showed that the optimal solution of the  single-ADR
maintenance problem is a threshold policy,  where it is optimal to send
the maintenance crew as soon the belief state drops below a
threshold. Using the structure of the single-ADR optimal policy we showed
that the Whittle index as a function of the belief state of an ADR can be
computed via a single binary search. 
We explored the performance of the Whittle-index policies
when the meter and utility clocks are (resp. not) synchronized and the load-shed
is random (resp. deterministic)  (see Section~\ref{sec:cases} for the details
of the four cases).
We also  developed a hierarchical
Bayesian method for computing the joint posterior distribution of the
load-shed and clock mismatch. 
Our numerical experiments suggest that, when the level of noise in the
meter readings is 
on average of the same size of the load shed, the Whittle-index policy is at
most $3.95$\%
suboptimal.
This problem and its solution was motivated by our research collaboration with
AutoGrid, a software provider for managing DR programs.

\bibliographystyle{unsrt}
\bibliography{references}

\vspace{-0.8cm}

\begin{biographynophoto}{Carlos Abad}
is a PhD candidate in the Industrial Engineering and Operation Research 
Department at Columbia University. His 
research 
interests include convex optimization,
compressive sensing, Bayesian inference, decision making under uncertainty, and 
their applications to the operation and economics of electric grids.
\end{biographynophoto}

\vspace{-0.8cm}

\begin{biographynophoto}{Garud Iyengar}
received a Ph.D. in Electrical Engineering from Stanford University. He joined 
Columbia University's Industrial Engineering and Operations Research Department 
in 1998
 and teaches courses in asset allocation, asset pricing, simulation and 
optimization. 
 His research interests include 
convex, robust and combinatorial optimization, queuing networks, mathematical 
and computational finance, communication and 
information theory.
\end{biographynophoto}

\end{document}

%% file: defs.tex
\def\br{\textbf{r}}

\def\bx{\textbf{x}}  
\def\by{\textbf{y}}
\def\bz{\textbf{z}}

\def\bI{\textbf{I}}

\def\bP{\textbf{P}}
\def\bQ{\textbf{Q}}
\def\bR{\textbf{R}}

\def\bV{\textbf{V}}

\def\bo{\textbf{0}}

\def\bfvarepsilon{{\boldsymbol{\varepsilon}}}

\def\bfomega{{\boldsymbol{\omega}}}

\def\cA{\mathcal{A}}
\def\cB{\mathcal{B}}

\def\cH{\mathcal{H}}

\def\cK{\mathcal{K}}

\def\cM{\mathcal{M}}
\def\cN{\mathcal{N}}
\def\cO{\mathcal{O}}

\def\cS{\mathcal{S}}

\def\cX{\mathcal{X}}

\def\mE{\mathbb{E}}

\def\mP{\mathbb{P}}

\def\smskip{\smallskip}
 
\def\texitem#1{\par\smskip\noindent\hangindent 25pt
               \hbox to 25pt {\hss #1 ~}\ignorespaces}

\def\norm#1{\left\|#1\right\|}

\newcommand{\BEAS}{\begin{eqnarray*}}
\newcommand{\EEAS}{\end{eqnarray*}}
\newcommand{\BEA}{\begin{eqnarray}}
\newcommand{\EEA}{\end{eqnarray}}
\newcommand{\BEQ}{\begin{eqnarray}}
\newcommand{\EEQ}{\end{eqnarray}}
\newcommand{\BIT}{\begin{itemize}}
\newcommand{\EIT}{\end{itemize}}
\newcommand{\BNUM}{\begin{enumerate}}
\newcommand{\ENUM}{\end{enumerate}}

\newcommand{\BA}{\begin{array}}
\newcommand{\EA}{\end{array}}

\newcommand{\ones}{\mathbf 1}

\newcommand{\reals}{{\mathbb{R}}}

\newcommand{\argmax}{\mathop{\rm argmax}}

\newcommand{\dif}[1]{\mathrm{d}#1}